\documentclass[12pt, 14paper,reqno]{amsart}
\usepackage{amsmath,amsfonts,amssymb}
\usepackage{longtable}

\usepackage[mathscr]{eucal}

\usepackage{amssymb, amsmath, amsthm}
\usepackage[breaklinks]{hyperref}

\usepackage{graphicx}
\usepackage[mathscr]{eucal}
\usepackage{mathrsfs}
\usepackage{amsbsy}
\usepackage{dsfont}
\usepackage{bbm}
\usepackage{wasysym}
\usepackage{stmaryrd}
\usepackage{url}

\input xypic
\xyoption{all}

\makeindex
\makeglossary

\begin{document}
\baselineskip = 16pt

\newcommand \ZZ {{\mathbb Z}}
\newcommand \NN {{\mathbb N}}
\newcommand \RR {{\mathbb R}}
\newcommand \PR {{\mathbb P}}
\newcommand \AF {{\mathbb A}}
\newcommand \GG {{\mathbb G}}
\newcommand \QQ {{\mathbb Q}}
\newcommand \CC {{\mathbb C}}
\newcommand \bcA {{\mathscr A}}
\newcommand \bcC {{\mathscr C}}
\newcommand \bcD {{\mathscr D}}
\newcommand \bcF {{\mathscr F}}
\newcommand \bcG {{\mathscr G}}
\newcommand \bcH {{\mathscr H}}
\newcommand \bcM {{\mathscr M}}
\newcommand \bcJ {{\mathscr J}}
\newcommand \bcL {{\mathscr L}}
\newcommand \bcO {{\mathscr O}}
\newcommand \bcP {{\mathscr P}}
\newcommand \bcQ {{\mathscr Q}}
\newcommand \bcR {{\mathscr R}}
\newcommand \bcS {{\mathscr S}}
\newcommand \bcV {{\mathscr V}}
\newcommand \bcW {{\mathscr W}}
\newcommand \bcX {{\mathscr X}}
\newcommand \bcY {{\mathscr Y}}
\newcommand \bcZ {{\mathscr Z}}
\newcommand \goa {{\mathfrak a}}
\newcommand \gob {{\mathfrak b}}
\newcommand \goc {{\mathfrak c}}
\newcommand \gom {{\mathfrak m}}
\newcommand \gon {{\mathfrak n}}
\newcommand \gop {{\mathfrak p}}
\newcommand \goq {{\mathfrak q}}
\newcommand \goQ {{\mathfrak Q}}
\newcommand \goP {{\mathfrak P}}
\newcommand \goM {{\mathfrak M}}
\newcommand \goN {{\mathfrak N}}
\newcommand \uno {{\mathbbm 1}}
\newcommand \Le {{\mathbbm L}}
\newcommand \Spec {{\rm {Spec}}}
\newcommand \Gr {{\rm {Gr}}}
\newcommand \Pic {{\rm {Pic}}}
\newcommand \Jac {{{J}}}
\newcommand \Alb {{\rm {Alb}}}
\newcommand \Corr {{Corr}}
\newcommand \Chow {{\mathscr C}}
\newcommand \Sym {{\rm {Sym}}}
\newcommand \Prym {{\rm {Prym}}}
\newcommand \cha {{\rm {char}}}
\newcommand \eff {{\rm {eff}}}
\newcommand \tr {{\rm {tr}}}
\newcommand \Tr {{\rm {Tr}}}
\newcommand \pr {{\rm {pr}}}
\newcommand \ev {{\it {ev}}}
\newcommand \cl {{\rm {cl}}}
\newcommand \interior {{\rm {Int}}}
\newcommand \sep {{\rm {sep}}}
\newcommand \td {{\rm {tdeg}}}
\newcommand \alg {{\rm {alg}}}
\newcommand \im {{\rm im}}
\newcommand \gr {{\rm {gr}}}
\newcommand \op {{\rm op}}
\newcommand \Hom {{\rm Hom}}
\newcommand \Hilb {{\rm Hilb}}
\newcommand \Sch {{\mathscr S\! }{\it ch}}
\newcommand \cHilb {{\mathscr H\! }{\it ilb}}
\newcommand \cHom {{\mathscr H\! }{\it om}}
\newcommand \colim {{{\rm colim}\, }} 
\newcommand \End {{\rm {End}}}
\newcommand \coker {{\rm {coker}}}
\newcommand \id {{\rm {id}}}
\newcommand \van {{\rm {van}}}
\newcommand \spc {{\rm {sp}}}
\newcommand \Ob {{\rm Ob}}
\newcommand \Aut {{\rm Aut}}
\newcommand \cor {{\rm {cor}}}
\newcommand \Cor {{\it {Corr}}}
\newcommand \res {{\rm {res}}}
\newcommand \red {{\rm{red}}}
\newcommand \Gal {{\rm {Gal}}}
\newcommand \PGL {{\rm {PGL}}}
\newcommand \Bl {{\rm {Bl}}}
\newcommand \Sing {{\rm {Sing}}}
\newcommand \spn {{\rm {span}}}
\newcommand \Nm {{\rm {Nm}}}
\newcommand \inv {{\rm {inv}}}
\newcommand \codim {{\rm {codim}}}
\newcommand \Div{{\rm{Div}}}
\newcommand \CH{{\rm{CH}}}
\newcommand \sg {{\Sigma }}
\newcommand \DM {{\sf DM}}
\newcommand \Gm {{{\mathbb G}_{\rm m}}}
\newcommand \tame {\rm {tame }}
\newcommand \znak {{\natural }}
\newcommand \lra {\longrightarrow}
\newcommand \hra {\hookrightarrow}
\newcommand \rra {\rightrightarrows}
\newcommand \ord {{\rm {ord}}}
\newcommand \Rat {{\mathscr Rat}}
\newcommand \rd {{\rm {red}}}
\newcommand \bSpec {{\bf {Spec}}}
\newcommand \Proj {{\rm {Proj}}}
\newcommand \pdiv {{\rm {div}}}
\newcommand \wt {\widetilde }
\newcommand \ac {\acute }
\newcommand \ch {\check }
\newcommand \ol {\overline }
\newcommand \Th {\Theta}
\newcommand \cAb {{\mathscr A\! }{\it b}}

\newenvironment{pf}{\par\noindent{\em Proof}.}{\hfill\framebox(6,6)
\par\medskip}

\newtheorem{theorem}[subsection]{Theorem}
\newtheorem{conjecture}[subsection]{Conjecture}
\newtheorem{proposition}[subsection]{Proposition}
\newtheorem{lemma}[subsection]{Lemma}
\newtheorem{remark}[subsection]{Remark}
\newtheorem{remarks}[subsection]{Remarks}
\newtheorem{definition}[subsection]{Definition}
\newtheorem{corollary}[subsection]{Corollary}
\newtheorem{example}[subsection]{Example}
\newtheorem{examples}[subsection]{examples}

\title{Branched covers of $\PR^1$ and divisibility in class group}
\author{Kalyan Banerjee, Kalyan Chakraborty and Azizul Hoque}

\address{K. Banerjee @Department of Mathematics, SRM University AP, Mangalagiri-Mandal, Amaravati-522502, Andhra Pradesh, India.}
\email{kalyan.ba@srmap.edu.in}
\address{K. Chakraborty @Department of Mathematics, SRM University AP, Mangalagiri-Mandal, Amaravati-522502, Andhra Pradesh, India.}
\email{kalyan.c@srmap.edu.in}
\address{Azizul Hoque @Department of Mathematics,  Gauhati University, Guwahati-781014, Assam, India.}
\email{azizul@gauhati.ac.in, ahoque.ms@gmail.com}

\date{\today}

\maketitle

\begin{abstract}
We start with $n$-torsions in the Jacobian of an $m$-gonal curve and produce $n$-torsions in the class group of certain number field $K$.
\end{abstract}

\section{Introduction}
One of the classical problem in algebraic number theory is to understand the structure of the class group of a given number field. It is known that the class group is finite for all number fields. The question is how to find an element of say order $n$ in the class group of the given number field. One of the approaches is to use algebro-geometric methods to find out such elements. This idea of using algebraic geometry to find elements of large order was first introduced in  \cite{AP} by Agboola and Pappus. Later,   Gillibert and Levin \cite{GL}  approached the problem by pulling back torsion line bundles to the class group. In \cite{GI}, Gillibert has shown how to pull back torsion elements on a hyperelliptic curve to the class group of quadratic number fields.

Our approach to this problem is to start with a branched cover of $\PR^1$ over $\QQ$ and spread it over $\Spec(\ZZ)$. Then we get a fibration say $\bcC\to \PR^1_{\ZZ}$, where $\bcC$ is a fixed spread of $C$ over $\Spec(\ZZ)$. Suppose we take a torsion element in  the Jacobian of $C$ defined over $\QQ$, and  spread this torsion divisor on $\bcC$ and consider restriction of this divisor on the fiber at a point $P\in \ZZ$, such that the fiber is smooth (that is we consider the good primes). Let us call it $\bcC_P$. The spectrum of the ring of integers $\bcO $  in the fiber $C_P$ corresponds to an affine open subset of $\bcC$.  We argue that this element in the corresponding class group of $\bcO$ is a nontrivial torsion element.

To accomplish this, we start with the theory of Chow schemes and Hilbert schemes for arithmetic varieties which parametrize cycles on an arithmetic variety, and then use \'Etale monodromy of a fibration to conclude that the torsion element mentioned above vary in a family. Precisely, we prove the following result. 

\begin{theorem}\label{thm1}
Let the Tate module $T_l(\bcC_{\bar\eta})$ be non-trivial, where $\bar\eta$ is the spectrum of a separably closed field (e.g. in case of $\bar{\QQ}$ the assumption implies $H^1(C_{\bar\QQ}, \QQ_l)$ is $\neq 0$). Then there exist infinitely number fields  $L$ whose class group  contain $l^i$ torsions for some $i$ and for some prime $l$.
\end{theorem}

\section{Proof of Theorem \ref{thm1}}
Let $C$ be a smooth projective curve defined over $\QQ$ with a $m:1$ regular map to $\PR^1_{\QQ}$. This map is branched and there exists no non-trivial \'etale cover of $\PR^1_{\QQ}$. We consider an integral model of $C$ and $\PR^1_{\QQ}$ over $\Spec(\ZZ)$. Then we have a morphism of $\Spec(\ZZ)$-schemes $\bcC\to \PR^1_{\ZZ}$, where $\bcC$ is the chosen integral model of $C$ over $\Spec(\ZZ)$. Then for a general  scheme-theoretic point $ {P}$ on $\PR^1_{\ZZ}$,  the normalization of the fiber over it, is the ring of integers of the field $f^{-1}(P)$, where $f$ is from $C\to \PR^1_{\QQ}$ and $P$ is considered as a general scheme theoretic point on $\PR^1_{\QQ}$. Let us consider a simple example. Suppose that $C$ is a smooth projective curve defined over $\QQ$ embedded into $\PR^2$, whose equation is given by
$$y^m=f(x)=a_0x^n+\cdots+a_n $$
where all $a_i$'s are rational numbers. Then we write
$$a_i=m_i/n_i$$
and put in the above equation. Then it becomes
$$y^m=\sum_i m_i/n_i x^{n-i}$$
Clearing out the denominator, we have
$$ny^m=\sum_i m_i'x^{n-i}\;.$$
Further we consider a general  point $p$ on $\AF^1_{\ZZ}\subset \AF^1_{\QQ}$. Putting $x=p$ for a general prime in the above equation we have
$$ny^m=\sum_i m_i'p^{n-i}$$
and we have an extension of $\ZZ$ given by
$$\ZZ[y]/ny^m-\sum_i m_i'p^{n-i}\;.$$
The integral closure of the above is the ring of integers of
$$K=\QQ[\sqrt[m]{f(t)}]\;.$$
This $p$ is chosen such that the above ring is a Dedekind domain and it's spectrum inherits a smooth scheme structure.

Now we consider some element in the Chow group of codimension one cycles on $C$ of degree zero (denoted by $A^1(C)$), that is isomorphic to the Jacobian variety or the Picard Variety of $C$. Let $\alpha$ be the element which is a non-trivial $n$-torsion on $A^1(C)$ which is naturally identified with  $J(C)$, the Jacobian of $C$.

Next we consider a fixed spread of the cycle $\alpha$ over $\PR^1_{\ZZ}$ denoted by $\tilde{\alpha}$ and consider the fiber of the spread at some general scheme theoretic point $P$, that is $\tilde{\alpha}_P$. This is a torsion element in the Chow group of the smooth arithmetic variety $\bcC_P$, which contains $\bcO_K$ as a Zariski open set. Restricting $\tilde{\alpha}_P$ to $\bcO_K$ gives rise to an element in the class group of $\bcO_K$.
Let $U\subset \Spec{\ZZ}$ be the set of all primes at which the fibers are  smooth.
Now consider the following subset
$$B=\{P\in U:\tilde{\alpha}_P=0\}\subset U\;,$$
and prove that:

\begin{theorem} The set
$$\bcZ_d:=\{(z,b)\in C^1_{d,d}(\bcC_U/U)|Supp(z)\subset {\bcC_b}, [z]=0\in \CH^1({\bcC_{b}})\}$$ is a countable union of Zariski closed subsets in the Chow variety $C^1_{d,d}(\bcC_U/U)$ parametrizing the pairs of degree $d$ subvarieties of the arithmetic variety $\bcC_U$.
\end{theorem}
\begin{proof}
This theorem in the case of a surface fibered over $\PR^2$  was proved in \cite{BH}, but we reconstruct the proof for the curve case for the convenience and completeness.
 There are some crucial points to be noted here:
\\
(I) The notion of Hilbert scheme and the Hom-scheme make sense for an arithmetic variety. This is as explained in \cite[Chapter: Hilbert schemes and Quot schemes, \S 5]{FGA}.
\\
(II) The family of Weil divisors of a smooth fibration over $\Spec(\ZZ)$ are parametrized by a Chow variety which is actually given by the Picard scheme parametrizing relative Cartier divisors of the same family \cite[Corollary 11.8]{Ry}. In our case the family is $\bcC_U$ which is of finite presentation over $\ZZ$ and it is a standard smooth algebra\footnote{in the sense, stack exchange \cite{St}, Definitions 10.136.6 and 29.32.1.} over $\ZZ$. This enables us to formulate the definition of rational equivalence for arithmetic varieties as in \cite[\S 3.3]{GS} in the following way:

Two Weil divisors $D_1,D_2$ are rationally equivalent on a fiber $\bcC_b$, if there exists a morphism $$f: \PR^1_{U}\to C^1_{d,d}(\bcC_U/\PR^1_{U})$$ such that
$(f\circ 0)|_b=D_1+B$ and $ (f\circ \infty)|_b=D_2+B$,
where $B$ is a positive Weil divisor and $0,\infty$ are two fixed sections from $U$ to  $\PR^1_{U}$.

Let us assume that the divisor $D_b=D_b^+-D_b^-$ is rationally equivalent to zero. This means that there exists a map $$f:\PR^1\to C^1_{d,d}({\bcC_b})$$ such that
$$f(0)=D_b^{+}+\gamma\text{ and }f(\infty)=D_b^{-}+\gamma,$$
where $\gamma$ is a positive divisor on ${\bcC_b}$.
In other words, we have the following map:
$$\ev:Hom^v(\PR^1_{U},C^1_{d}(\bcC_U/U))\to C^1_{d}(\bcC_U/U)\times C^1_{d}(\bcC_U/U) $$
 given by $f\mapsto (f(0),f(\infty))$ and that the fiber of  $f$ at $b$ is contained in $C^1_{d,d}(\bcC_b)$.

Let us denote $C^1_{d}(\bcC_U/U)$ by $C^1_d(\bcC)$ for simplicity.
We now consider the subscheme $U_{v,d}(\bcC)$ of $\PR^1_{U}\times \Hom^v(\PR^1_{U},C^1_{d}(\bcC))$ consisting of the pairs $(b,f)$ such that image of $f$ is contained in $C^1_{d}(\bcC_b)$ (such a universal family exists, for example see \cite[Theorem 1.4]{Ko} or \cite{FGA}, [Chapter on Hilbert schemes and Quot schemes]. This gives a morphism from $U_{v,d}(\bcC)$ to
$$\PR^1_{U}\times C^1_{d,d}(\bcC_b)$$
 defined by $$(b,f)\mapsto (b,f(0),f(\infty)).$$
Again, we consider the closed subscheme $\bcV_{d,d}$ of $\PR^1_{U}\times C^1_{d,d}(\bcC)$ given by $(b,z_1,z_2)$, where $(z_1,z_2)\in C^1_{d,d}(\bcC_b)$. Suppose that the map from $\bcV_{d,u,d,u}$ to $\bcV_{d+u,u,d+u,u}$ is given by
$$(A,C,B,D)\mapsto (A+C,C,B+D,D).$$
Let us denote the fiber product by $\bcV$ of $U_{v,d}(\bcC)$ and $\bcV_{d,u,d,u}$ over $\bcV_{d+u,u,d+u,u}$. If we consider the projection from $\bcV$ to $\PR^1_{U}\times C^1_{d,d}(\bcC)$, then we observe that $A$ and $B$ are supported as well as rationally equivalent on $\bcC_b$. Conversely, if $A$ and $B$ are supported as well as rationally equivalent on $\bcC_b$, then we get the map $$f:\PR^1_{U}\to C^1_{d+u,u,d+u,u}(\bcC)$$ of some degree $v$ satisfying
$$f(0)=(A+C,C)\text{ and } f(\infty)=(B+D,D),$$
where $C$ and $D$ are supported on $\bcC_b$. This implies that the image of the projection from $\bcV$ to $\PR^1_{U}\times C^1_{d,d}(\bcC)$ is a quasi-projective subscheme $W_{d}^{u,v}$ consisting of the tuples $(b,A,B)$ such that $A$ and $B$ are supported on $\bcC_b$, and that there exists a map $$f:\PR^1_{U}\to C^1_{d+u,u}(\bcC_b)$$ such that
$$f(0)=(A+C,C)$$
 and
 $$f(\infty)=(B+D,D)\;.$$
  Here $f$ is of degree $v$, and $C,D$ are supported on $\bcC_b$ and they are of co-dimension $1$ and degree $u$ cycles. This shows that $W_d$ is the union $\cup_{u,v} W_d^{u,v}$. We now prove that the Zariski closure of $W_d^{u,v}$ is in $W_d$ for each $u$ and $v$. For this, we prove the following:
$$W_d^{u,v}=pr_{1,2}(\wt{s}^{-1}(W^{0,v}_{d+u}\times W^{0,v}_u)),$$
where
$$\wt{s}: \PR^1_{U}\times C^1_{d,d,u,u}(\bcC)\to \PR^1_{U}\times C^1_{d+u,d+u,u,u}(\bcC)$$
defined by
$$\wt{s}(b,A,B,C,D)=(b,A+C,B+D,C,D).$$

We assume $(b,A,B,C,D)\in \PR^1_{U}\times C^1_{d,d,u,u}(\bcC)$ in such a way that $\wt{s}(b,A,B,C,D)\in W^{0,v}_{d+u}\times W^{0,v}_u$. This implies that there exists an element
$$(b,g)\in \PR^1_{U}\times\Hom^v(\PR^1_{U},C^p_{d+u}(\bcC))$$ and an element
$$(b,h)\in \Hom^v(\PR^1_{U},C^p_{u}(\bcC))$$ satisfying $$g(0)=A+C,~g(\infty)=B+D \text{ and } h(0)=C,h(\infty)=D$$ as well as the image of $g$ and $h$ are contained in $C^1_{d+u}(\bcC_b)$ and  $C^1_u(\bcC_b)$ respectively.

Also if $f=g\times h$ then $f\in \Hom^v(\PR^1_{U},C^p_{d+u,u}(\bcC))$ such that the image of $f$ is contained in $C^1_{d+u,u}(\bcC_b)$ as well as it satisfies the following:
$$f(0)=(A+C,C)\text{ and }(f(\infty))=(B+D,D).$$
This shows that $(b,A,B)\in W^d_{u,v}$.

On the other hand, if we assume that $(b,A,B)\in W^d_{u,v}$, then there exists $f\in \Hom^v(\PR^1_U,C^1_{d+u,u}(\bcC))$ such that
$$f(0)=(A+C,C)\text{ and }f(\infty)=(B+D,D),$$
and image of $f$ is contained in the Chow scheme of  $\bar{\bcC_b}$.

We now compose $f$ with the projections to $C^1_{d+u}(\bcC_b)$ and to $C^1_{u}(\bcC_b)$ to get a map $g\in \Hom^v(\PR^1_{U},C^1_{d+u}(\bcC))$ and a map $h\in\Hom^v(\PR^1_{U},C^1_{u}(\bcC))$ satisfying
$$g(0)=A+C,\quad g(\infty)=B+D$$
and
$$h(0)=C,\quad h(\infty)=D.$$
Also, the image of $g$ and $h$ are contained in the respective Chow varities of the fibers $\bcC_b$. Therefore, we have
$$W_d=pr_{1,2}(\wt{s}^{-1}(W_{d+u}\times W_u)).$$

We are now in a position to prove that the closure of $W_d^{0,v} $ is contained in $W_d$. Let $(b,A,B)$ be a closed point in the closure of ${W_d^{0,v}}$. Let $W$ be an irreducible component of ${W_d^{0,v}}$ whose closure contains $(b,A,B)$. We assume that $U'$ is an affine neighborhood of $(b,A,B)$ such that $U'\cap W$ is non-empty. Then there is an irreducible curve $C'$ in $U'$ passing through $(b,A,B)$. Let  $\bar{C'}$ be the Zariski closure of $C'$ in $\overline{W}$. The map
$$e:U_{v,d}(\bcC)\subset \PR^1_{U}\times \Hom^v(\PR^1_{U},C^1_{d}(\bcC))\to C^1_{d,d}(\bcC)$$
given by
$$(b,f)\mapsto (b,f(0),f(\infty))$$
is regular and $W_d^{0,v}$ is its image. We now choose a curve $T$ in $U_{v,d}(\bcC)$ such that the closure of $e(T)$ is $\bar C'$.  Let $\wt{T}$  denote the normalization of the Zariski closure of $T$, and $\wt{T_0}$ be the pre-image of $T$ in this normalization. Then the regular morphism $\wt{T_0}\to T\to \bar C'$ extends to a regular morphism, when scalar-extended to the field of algebraic numbers. Let this morphism be $\wt{T}_{\QQ}$ to $\bar C'_{\QQ}$. If $(b_{\QQ},f_{\QQ})$ is a pre-image of $(b_{\QQ},A_{\QQ},B_{\QQ})$, then
$$f_{\QQ}(0)=A_{\QQ}, \quad f_{\QQ}(\infty)=B_{\QQ}$$ and the image of $f_{\QQ}$ is contained in $C^p_{d}(C)$. 
Spreading out $f_{\QQ}$, we have an $f$ such that $$f(0)=A, \quad f(\infty)=B\;.$$
This is because there is a one to one correspondence between $\Spec(\bar\ZZ)$-points of arithmetic varieties and $\bar Q$-points of the corresponding variety over $\bar Q$. Therefore, $A$ and $B$ are  rationally equivalent. This completes the proof.
\end{proof}

\subsection{$l^i$ torsion points on the class groups for some prime $l$ and some positive integer $i$}
In this section, we prove Theorem \ref{thm1}.

\begin{proof}[Proof of Theorem \ref{thm1}]
Let us consider the fibration $\bcC\to \PR^1_{\ZZ}$. Restricting this fibration to some Zariski open set $U$ in $\PR^1_{\ZZ}$, we have an \'etale morphism from $f_U:\bcC_U\to U$. Now consider the higher direct image sheaf $R^1f_{U*}\ZZ/l^n\ZZ$, where $\ZZ/l^n\ZZ$ denotes the constant sheaf on $U$ and $f_{U*}$ denote the pushforward at the level of sheaves. Then it is well known that this sheaf
$$R^1f_{U*}\ZZ/l^n\ZZ$$
is a locally constant, constructible sheaf. Also its fiber over the geometric  point $\bar\eta$ of $U$ is isomorphic to the \'etale cohomology
$$H^1(\bcC_{\bar\eta},\ZZ/l^n\ZZ)$$
and that, this cohomology group is a $\pi_1(U,\bar\eta)$ module, where $\pi_1$ stands for the \'etale fundamental group of $U$. Therefore taking the inverse limit with varying $n$, we have an action of $\pi_1(U,\bar\eta)$ on $H^1(\bcC_{\bar\eta},\ZZ_l)$, where $\ZZ_l$ denote the group of $l$-adic integers. Taking the tensor product with $\QQ_l$, we have $H^1(\bcC_{\bar\eta},\QQ_l)$ is a $\pi_1(U,\bar\eta)$ module, which is dual to the Tate module $T_l(\bcC_{\bar\eta})$.

Now suppose  $T_l(\bcC_{\bar\eta})$ is a non-zero $\pi_1$ module.  Therefore for  $\bar\eta$, there exists some $l^i$ torsion on $J(\bcC_{\bar\eta})$ (after having a finite extension of $\QQ$). Since $T_l(\bcC_{\bar\eta})$ is a non-zero $\pi_1$-module it corresponds to a locally constant sheaf in some Zariski open neighborhood $U'$ of $\bar\eta$. Considering any other geometric point $\bar\eta'$ in $U'$, the Tate module $T_l(\bcC_{\bar\eta'})$ is also of the same rank as $T_l(\bcC_{\bar\eta})$ as a $\QQ_l$ module.

Now $\CH^1(\bcC_{\bar\eta})$ is the colimit of $\CH^1(V)$, for $V$ Zariski open in $\bcC_{U'}$. Therefore there is a Zariski open set $V$ in $\bcC_{\bar\eta}$, such that there exists $l^i$-torsion in $\CH^1(V)$. Now given a $P$ in $V$, over which we have $\bar\eta'$, the $l^i$ torsion restricted to $\CH^1(V_P)$ is non-zero, because if it is zero then its pullback to $\CH^1(\bcC_{\bar\eta'})$ is zero, which is the non-zero element we started with. Such a $V_P$ is  $\Spec(\bcO_L)$, where $L$ is a finite extension of $K$  mentioned above. Therefore there exists $l^i$ torsions in $\CH^1(\Spec(\bcO_L))$, which is the class group of $\bcO_L$ (this follows from localization exact sequence). Considering the presence of \'etale monodromy it follows that infinitely many number fields of a fixed degree has an element of order $l^i$.
\end{proof}

\section{An example} Let us consider the super-elliptic curve given by $y^5=x^5-31$. Since the Tate module of this curve is non-trivial, by Theorem \ref{thm1}, there exist infinitely many number fields which are finite extensions of  $K=\QQ(\zeta_5)$, the $5$-th degree cyclotomic field,  such that the class group of the finite extension  of $K$ (which is $L$) is divisible by $l^i$ for some $i$ and some prime $l$. In particular, if we take $l=5$ and the degree of the extension of $[L:K]$ is not divisible by $5$ then the class group of the finite extension surjects onto the class group of  $K=\QQ(\zeta_5)$  and we get $5$-divisibility of the class group of $\QQ(\zeta_5)$. 

In fact for any prime $p$, there exists a finite extension $L$ of $\QQ(\zeta_p)$ such that the class number of $L$ is divisible by $p$. Such an extension can be obtained from the super-elliptic curve 
$y^p=x^p-(2^p-1)$ (specializing at the prime $2$). 

If $p$ does not divide  $[L:\QQ(\zeta_p)]$, then $p$ divides the class number of $\QQ(\zeta_p)$. Hence by the result of Herbrandt-Ribet \cite{Ri} $p$ must divide at least one of the Bernoulli numbers $B_k$, where $2\leq k \leq p-3$ and $k$ even.

\noindent\textbf{Acknowledgements.}
This work is partially supported by SERB MATRICS grant (No. MTR/2021/000762) and SERB  CRG grant (No. CRG/2023/007323), Govt. of India.

\end{document}